\documentclass{article}
\usepackage{amsmath, amssymb, amsthm, graphics}
\usepackage[T1]{fontenc}
\usepackage[latin2]{inputenc}

\newtheorem{lem}{Lemma}
\newtheorem{tet}{Theorem}

\newtheorem{corol}{Corollary}
\newtheorem{prop}{Proposition}

\theoremstyle{remark}
\newtheorem{remark}{Remark}
\newtheorem{pelda}{Example}
\theoremstyle{definition}
\newtheorem{defi}{Definition}

\newcommand*{\uos}[2]{\ensuremath{ \{ #1 :#2 \}}}
\newcommand*{\os}[2]{\ensuremath{ ( #1 : #2 )}}
\newcommand*{\egy}{\mbox{\boldmath$1$}}
\newcommand*{\nulla}{\mbox{\boldmath$0$}}

\newcommand*{\dimz}{\mbox{dim}}

\newcommand*{\Span}{\mbox{Span}}

\newcommand*{\sop}{\kappa}
\newcommand*{\sopp}{\underline{\sop}}
\newcommand*{\mme}{\perp_{\cap}}
\fontsize{11}{16}\selectfont
\begin{document}
\title{Conditional independence relations and log-linear models for random permutations}
\author{Vill\H o Csisz\'ar \thanks{Lor\'and E\"otv\"os University, Budapest,
e-mail: csvillo@gmail.com}}
\maketitle

\begin{abstract}
We propose a new class of models for random permutations, which we
call log-linear models, by the analogy with log-linear models used 
in the analysis of contingency tables. 
As a special case, we study the family of all Luce-de\-com\-po\-sable
distributions, and the family of those random permutations, for which
the distribution of both the permutation and its inverse is
Luce-de\-com\-po\-sable. We show that these latter models can be described
by conditional independence relations. We calculate the number of free parameters
in these models, and describe an iterative algorithm for maximum 
likelihood estimation, which
enables us to test if a set of data satisfies the conditional
independence relations or not.   


\end{abstract}

\section{Introduction}
There are three slightly different
situations, in which permutation-valued data may turn up.

$(i)$ The permutation describes a \emph{pairing}, i.e. a one to one
correspondance, between two sets of cardinality $n$, whose elements
are labelled with the numbers $1,\ldots, n$.
An example is a pairing of boys and girls for a dance. If the first
set is the boys and the second set is the girls, then a pairing
is given by the permutation $\pi$, where $\pi(i)=j$ means that boy
$i$ dances with girl $j$. 

$(ii)$ The permutation describes a \emph{ranking} of a labelled
set of cardinality $n$, i.e. the ranking from best to worst of
labelled alternatives. The ranking is given by $\pi$, where $\pi(i)=j$ means that alternative $i$ is ranked
$j$th best. The inverse of the ranking $\pi$ is
the \emph{ordering} $\pi^{-1}$, i.e. $\pi^{-1}(i)$ is the label of the alternative ranked $i$th best.

$(iii)$ The permutation describes a \emph{reordering} of a set of
ordered elements. For
example, books on a shelf in a library are reordered as readers look
into them. Here $\pi(i)=j$ means that the $i$th item in the original
order becomes the $j$th item after reordering.

Notice that a pairing or a ranking/ordering can be described by a permutation
only after a labelling is fixed on the sets. 

There is a vast literature of models for random permutations,
especially for ranking/ordering data $(ii)$. Comprehensive summaries
can be found, among others, in
Critch\-low, Fligner, and Verducci \cite{CFV} and Marden \cite{Marden}.
The larger classes of models are order statistics models (also called
Thurstonian models), distance-based models, paired
comparison models, and multistage models. 

Our starting point is the concept of Luce-de\-com\-po\-sability, also called
$L$-de\-com\-po\-sability (we will use the latter name). This property,
applied to orderings, was introduced by Critchlow, Fligner and Verducci in \cite{CFV}, motivated
by Luce's ranking postulate \cite{Luce59}. This postulate supposes
that the ordering of the alternatives is the result of
repeated selections of the best alternative from the remaining set of
alternatives. That is, for each set $C$ and each alternative $x\in C$,
the probability that $x$ is chosen as best from $C$ is given by
$p_C(x)$. Given these choice probabilities, the probability of the
ordering $\pi$ is given by
\[ p(\pi) = \prod_{k=1}^n p_{C_k}(\pi(k)),\]
where $C_k = \{ \pi(k), \ldots, \pi(n)\}$ is the set of available
alternatives at the $k$th step.
Luce combined this postulate with his choice axiom to develop the
Luce model. The choice axiom puts restrictions on the choice
probabilities $p_C(x)$. The ranking postulate without the choice axiom
produces a multistage model, a general $L$-de\-com\-po\-sable distribution. 
In other words, a random
ordering (or its distribution) is $L$-de\-com\-po\-sable, if its elements are chosen successively, satisfying the following
Markov property: in the $k$th step, the $k$th element is chosen from the
remaining ones, independently of the order of the first $k-1$
elements.  

In this paper, we wish to apply $L$-de\-com\-po\-sability to random
pairings, rankings, and reorderings as well. We notice that
$L$-de\-com\-po\-sability of a random pairing depends on the labelling of
the first set, and $L$-de\-com\-po\-sability of a random ranking depends on
the labelling of the alternatives. 
In fact, the labelling under which $L$-de\-com\-po\-sability is satisfied
(if such a labelling exists) can be interpreted as a natural order of the
elements or alternatives. If $n$ is relatively small, this natural
order may be found by an exhaustive search over all labellings.   

In the main part of the paper, we study random permutations $\Pi$, for which the distributions of both
$\Pi$ and $\Pi^{-1}$ are $L$-de\-com\-po\-sable. In this case we say that the
distribution of $\Pi$ (and of $\Pi^{-1}$) is bi-de\-com\-po\-sable. Bi-de\-com\-po\-sability of a random
\emph{pairing} implies a natural order on both sets, while
bi-de\-com\-po\-sability of a random \emph{ranking/ordering} implies a
natural order on the set of alternatives. This new concept is perhaps
most natural for random pairings, since they possess an obvious
symmetry in the two sets whose elements are paired, that is in $\Pi$
and $\Pi^{-1}$. The model is also attractive in the case of
rankings/orderings, since, as we shall show, a general
$L$-de\-com\-po\-sable distribution has $2^n(n/2 -1)+1$ free parameters,
while a general bi-de\-com\-po\-sable distribution possesses only
$\sum_{k=1}^{n-1} k^2$ parameters. This is still more than the usual
number of about $n$ or at most $n^2$ for most known models, however,
it is still very small compared to $n!$.

Another feature of both the $L$-de\-com\-po\-sable and the bi-de\-com\-po\-sable
families is that they can be characterised by certain conditional
independence relations, which we will describe later in detail. 
Therefore, by fitting these models with the maximum likelihood method,
and assessing the goodness of fit with the chi-square test, we can
test the hypothesis that the random permutation satisfies certain
conditional independence properties.

The paper is organised as follows. Section 2 deals with
$L$-de\-com\-po\-sability, and
Section 3 contains the main results about bi-de\-com\-po\-sable
distributions. In Section 3.1, we study general log-linear models for
random permutations, which we apply in Section
3.2 to de\-com\-po\-sable models, and prove Theorem
\ref{metszet}. Section 3.3 treats the problem
of maximum likelihood estimation in the models. It is shown that the maximum likelihood estimate
is explicite in the $L$-de\-com\-po\-sable model, but in the bi-de\-com\-po\-sable
model, it can only be obtained by iterative methods. In Section 4, we
investigate which models formulated in the literature are $L$-de\-com\-po\-sable or
bi-de\-com\-po\-sable. In Section 5 we study to what extent
the latent order with respect to which the distribution is
de\-com\-po\-sable can be estimated. Finally,
in Section 6, we fit the models to a real dataset, and Section 7
contains the proofs of some lemmas.

\section{$L$-de\-com\-po\-sability}
For integers $i\leq j, \; \{i:j\}$ denotes the set $\{k:\; i\leq k\leq j\}.$
For any vector $v=(v(1), \ldots, v(s))$, we call $v(i)$ the $i$th
element of $v$. For the set of the $i$th to $j$th elements, and for
the subvector of the $i$th to $j$th elements of $v$, introduce the notations
\begin{equation}  v\uos{i}{j}= \{v(i),\ldots, v(j)\}, \quad v\os{i}{j} = (v(i),\ldots,v(j)), \quad
1\leq i\leq j\leq s. \end{equation}
If $j<i$ then let $v\{i:j\}$ be the empty set.
Let $S_n$ stand for the symmetric group of all
permutations $\pi$ of $\{1, \ldots, n\}$.
We denote a probability distribution on $S_n$ by $p=\{ p(\pi): \pi\in S_n
\}$.  Denote by $\Pi: \Omega \to S_n$ a random permutation on a probability space
 $(\Omega, \mathcal{A},P)$ with distribution $p$, that
 is $P(\Pi = \pi) = p(\pi)$.
The idea of \emph{$L$-de\-com\-po\-sability} first appears in \cite{CFV}, and
was motivated by Luce's ranking postulate \cite{Luce59}. It states
that for any $k$, the value of $\Pi(k+1)$ depends on $\Pi\os{1}{k}$
only through $\Pi\uos{1}{k}$. Recall that the probability of a
permutation can always be written in the product form
\begin{equation} \label{szorzat} P(\Pi = \pi) = \prod_{k=0}^{n-1} P\left( \Pi(k+1)=\pi(k+1) \mid
 \Pi\os{1}{k}=\pi\os{1}{k}\right).\end{equation}
$L$-de\-com\-po\-sability means that the conditions
$\Pi\os{1}{k}=\pi\os{1}{k}$ can be replaced by the conditions $\Pi\uos{1}{k}=\pi\uos{1}{k}$.
We formulate this in the following definition, in four different
forms. For two permutations, $\pi\sigma$ denotes composition,
i.e. $\pi\sigma(i) = \pi(\sigma(i))$.

\begin{defi} \label{s1} Let $\Pi$ be a random permutation with probability distribution $p$ on $S_n$.
$\Pi$ or $p$ is called \emph{$L$-de\-com\-po\-sable}, if any of the following are
 satisfied.
\begin{enumerate}
\item
For every $2\leq k \leq n-2$, $\pi\in S_n$
 and $\sigma\in S_{k}$
\begin{multline}
\label{s2}
P\left( \Pi(k+1)=\pi(k+1) \mid  \Pi\os{1}{k}=\pi\os{1}{k}\right) =\\
= P\left(\Pi(k+1)=\pi(k+1) \mid \Pi\os{1}{k}=\pi\sigma\os{1}{k}\right),
\end{multline}
if both conditional probabilities are defined.
\item For every $2\leq k \leq n-2$ and $\pi\in S_n$
\begin{multline}
\label{s3}
P\left(\Pi(k+1)=\pi(k+1) \mid \Pi\os{1}{k}=\pi\os{1}{k}\right) = \\
= P\left(\Pi(k+1)=\pi(k+1) \mid \Pi\uos{1}{k} = \pi\uos{1}{k} \right),
\end{multline}
if the lefthandside is defined.
\item The random sets $\Pi\uos{1}{k}$ form a Markov chain
for $k=1,\ldots, n$.
\item
 There exists a $\Lambda$ nonnegative function defined on pairs
\begin{equation}\label{parok} (x,C): \, C\subset \{1,\ldots, n\}, \, x\not\in C,\end{equation}
 and $c$ constant, such that for all $\pi \in S_n$
\begin{equation}\label{s1.0}p(\pi) = c \prod_{k=0}^{n-1} \Lambda (\pi(k+1), \pi\uos{1}{k}).
\end{equation}
\end{enumerate}
\end{defi}

\begin{prop} The four properties in Definition \ref{s1} are equivalent.
\end{prop}
This proposition, in slightly different form, can be found in
\cite{CFV}, so we omit the straightforward proof. In the first two
equivalent forms of the definition, we could formally include $k=0,1,n-1$ as
well, but equations \eqref{s2} and \eqref{s3} are always satisfied for
these $k$-values. It follows that for $n\leq 3$, all distributions are
$L$-de\-com\-po\-sable.

The pair $(\Lambda, c)$ is called an $L$-decomposition of the distribution
$p$ if \eqref{s1.0} holds. By \eqref{szorzat} and \eqref{s3}, one $L$-decomposition of the $L$-de\-com\-po\-sable
distribution $p$ is given by $c=1$ and
\begin{equation} \label{s3.5} \Lambda(x,C) = P(\Pi(|C|+1) = x \mid \Pi\uos{1}{|C|}
= C),
\end{equation}
if the probability of the condition is positive,
otherwise $\Lambda(x,C)=0$. We call this $L$-decomposition \emph{canonical}.

The fact that the random sets $\Pi\uos{1}{k}$ form a Markov
chain is equivalent to the independence of the past and the future,
conditional on the present. This means that the first $k$ and last
$n-k$ elements of $\Pi$ are conditionally independent on the
condition that the set of the first $k$ elements is given. By the
well-known property of Markov chains, this observation generalises to
any consecutive partition of the set $\{1,\ldots,n\}$. 

Any $j$-tuple $\sop = (\sop_1, \ldots, \sop_j)$ with $\sop_0 =0< \sop_1 <
\ldots < \sop_j < n =\sop_{j+1}$ is a set of \emph{sections}, which define
a $\sopp$ \emph{consecutive partition} of the set $\{1,\ldots,n\}$ into
$j+1$ sets by
\begin{equation} \label{sop} \sopp = (\sopp_1, \ldots, \sopp_{j+1}), 
 \mbox{ where } \sopp_i = \uos{\sop_{i-1}+1}{\sop_{i}}.
\end{equation}
For the consecutive partition $\sopp$ and $\pi\in S_n$, define the
vector of unordered marginals
\begin{equation}\label{unordered m} \{\pi_{\sopp}\} = ( \{\pi_{\sopp_i}\}  :1\leq i\leq j+1), \mbox{ where
} \{\pi_{\sopp_i}\} = \pi\uos{\sop_{i-1}+1}{\sop_{i}} .\end{equation}
In contrast, $\pi_{\sopp_i} = \pi\os{\sop_{i-1}+1}{\sop_{i}}$ is an
ordered marginal.

$L$-de\-com\-po\-sability means that for any consecutive partition $\sopp$
in \eqref{sop} and any $\pi\in S_n$,
\begin{equation}\label{markov} P(\Pi=\pi \mid \{\Pi_{\sopp}\} =\{\pi_{\sopp}\} ) = \prod_{i=1}^{j+1}
P(\Pi_{\sopp_i}=\pi_{\sopp_i} \mid \{\Pi_{\sopp}\}
= \{\pi_{\sopp}\}).\end{equation}
Thus we have proved the following
\begin{prop}\label{lfeltfugg}
A random permutation $\Pi$ is $L$-de\-com\-po\-sable if and only if, for
every consecutive partition in \eqref{sop}, the ordered marginals
$\Pi_{\sopp_i}$, $1\leq i\leq j+1$ are conditionally independent,
given $\{\Pi_{\sopp}\}$, that is, \eqref{markov} holds.
\end{prop}

In the language of orderings, the unordered marginal $\{\pi_{\sopp}\}$
is the partial ordering of shape $\sopp$ derived from the ordering
$\pi$. This gives the set of alternatives receiving the first $\sop_1$
ranks, the set of alternatives receiving the next $\sop_2-\sop_1$
ranks, etc. Thus $L$-decomposability for orderings means that given
the partial ordering of arbitrary shape $\sopp$, the orderings within
each set of ranks $\sopp_i$ are independent.

\section{Bi-de\-com\-po\-sability}
We are interested in random permutations, for which the
distribution of both $\Pi$ and $\Pi^{-1}$ is $L$-de\-com\-po\-sable. 
If the
distribution of $\Pi$ is $p$, then the distribution of $\Pi^{-1}$ is
given by $p'(\pi) = p(\pi^{-1})$. Thus, $\Pi^{-1}$ is
$L$-de\-com\-po\-sable if and only if
\begin{equation} p(\pi) = p'(\pi^{-1}) = \prod_{(x,C)} \Lambda'(x,C)
 ^{m_L(\pi^{-1},(x,C)) }= \prod_{(x,C)} \Lambda'(x,C)
 ^{m_{L'}(\pi,(x,C)) } ,\end{equation}
where $\Lambda'$ is the canonical $L$-decomposition of $p'$, and the
 matrix $M_{L'}$ is derived from the matrix $M_L$ by interchanging each pair of rows
corresponding to inverse permutations $\pi$ and $\pi^{-1}$. We call such
 distributions $L'$-de\-com\-po\-sable, and denote their family by
 $\mathcal{P}_{L'}$.
The family of bi-de\-com\-po\-sable distributions will be denoted by
$\mathcal{P}_b = \mathcal{P}_{L}\cap   \mathcal{P}_{L'}$.

According to Proposition \ref{lfeltfugg}, bi-de\-com\-po\-sable random permutations have
the property that for every consecutive partition $\sopp$ in \eqref{sop}, the ordered marginals
$\Pi_{\sopp_i}$, $1\leq i\leq j+1$ are conditionally independent,
given $\{\Pi_{\sopp}\}$, and the ordered marginals
$\Pi^{-1}_{\sopp_i}$, $1\leq i\leq j+1$ are conditionally independent,
given $\{\Pi^{-1}_{\sopp}\}$. We now show that bi-de\-com\-po\-sable random
permutations satisfy additional conditional
independence statements. Let $\sopp$ and $\underline{\lambda}$ be two
consecutive partitions. For
$\pi\in S_n$, define the ordered marginals
\begin{equation} 
\pi_{\sopp_i \times \underline{\lambda}_j} = \{ (a, \pi(a)):\, a\in
\sopp_i, \, \pi(a)\in \underline{\lambda}_j \}.\end{equation}
We prove the next proposition in Section \ref{appendix}.
\begin{prop} \label{teglafugg} A  random permutation $\Pi$ is bi-de\-com\-po\-sable, if and
  only if for all pairs of consecutive partitions $\sopp$ and
  $\underline{\lambda}$ of sizes $s$ and $t$ respectively, the ordered marginals 
$\Pi_{\sopp_i \times \underline{\lambda}_j}$
for $1\leq i\leq s$ and $1\leq j\leq t$ are conditionally
independent, given $ \{\Pi_{\sopp} \}$ and $ \{
\Pi^{-1}_{\underline{\lambda}} \}$. 
\end{prop}

In the rest of the paper, we focus our attention on strictly positive
distributions, i.e. the case when $p(\pi)>0$ for all $\pi\in S_n$, which can be described as exponential families, or more
specifically, as log-linear models.
In general, by an exponential family of discrete (strictly positive) distributions
$p=(p_1,\ldots,p_s)$ we mean the family
\begin{equation} \label{exp_m}
\left\{ p: \, \sum_{i=1}^s p_i =1, \, \log p \in U \right\},
\end{equation}
where $U$ is a $t$-dimensional linear subspace of $\mathbb{R}^s$,
containing the vector $\egy = (1,\ldots,1)^T$. The number of free
parameters of the exponential family \eqref{exp_m} is $t-1$.
Extending the concept used in the analysis of
contingency tables, we may define a log-linear model as
an exponential family where the linear subspace $U$ has a generating
set consisting of $0-1$ vectors. All log-linear models for random
permutations appearing in this paper have the additional property that
the generating $0-1$ vectors of $U$ are indicator vectors of
different values of generalised marginals $|\pi_{\mathcal{B}}|$.

Denote by $\mathcal{P}_L^+$ and $\mathcal{P}_{L'}^+$ the family of
strictly positive $L$-decomposable and $L'$-decomposable distributions
respectively. Then the family of strictly positive bi-decomposable
distributions is given by $\mathcal{P}_b^+ = \mathcal{P}_L^+ \cap 
\mathcal{P}_{L'}^+$.  
We will show that both $\mathcal{P}_L^+$ and $\mathcal{P}_{L'}^+$
admit a log-linear representation with corresponding linear subspaces
$F$ and $G$ respectively.
It follows that $\mathcal{P}^+_b$ is also an exponential family with subspace
$H = F \cap G$. We will show that $\mathcal{P}^+_b$ is also a log-linear model, and we
 will determine the dimension and a basis of $H$
 (Theorem \ref{metszet}). This is made possible by the
 abundance of orthogonality. Two subspaces $U$ and $V$ of a
 Hilbert space are called
 orthogonal, if every pair of vectors $u\in U, v\in V$ are
 orthogonal. The closed subspaces $U$ and $V$ intersect each other
 orthogonally, if the (orthogonal) projection of $U$ on $V$ equals
 $U\cap V$, or equivalently, the projection of $V$ on $U$ equals
 $U\cap V$. Denote the operator of orthogonal projection on $U$ by
 $Pr_U$. Another equivalent condition for orthogonal intersection is
 that the projection operators $Pr_U$ and $Pr_V$ commute.
 Thus introducing the notation $\mme$ for orthogonal intersection,
\begin{equation}
U \mme V \iff Pr_U V = U\cap V \iff Pr_V U = U\cap V \iff Pr_U Pr_V =
Pr_V Pr_U.
\end{equation}

 We will show that $F$ and $G$
 intersect each other orthogonally, furthermore, we will find an
 orthogonal decomposition of both $F$ and $G$ into lower dimensional
 subspaces $F_k$ and $G_{\ell}$, such that each pair of subspaces
 $(F_k,G_{\ell})$ intersect each other orthogonally. Then it will
 suffice to determine the dimension and basis of the low dimensional
 subspaces $F_k \cap G_{\ell}$. Orthogonal intersection does not
 appear by coincidence, it is the consequence of conditional
 independence relations, as we will explain later on. Before carrying out this program in
 the following subsections, we state the main theorem of this section.

\begin{tet} \label{metszet} The
family of positive bi-de\-com\-po\-sable distributions is a log-linear
model with the number of free parameters equal to
\begin{equation} \label{metszetdim} d_n = \sum_{i=1}^{n-1} i^2.\end{equation}
 \end{tet}

The proof of Theorem \ref{metszet} is given in Section 3.2.

\subsection{Partitions of the chessboard}
A permutation may be identified with a placement of $n$ rooks on the
$n\times n$ chessboard such that they cannot capture each other, i.e.
a placement with exactly one rook in each row and in each
column. Let us agree that we place the rooks ``row-wise'', that is
if $\pi(i) = j$, then we place a rook in the $j$th square of the
$i$th row. $\pi^{-1}$ can be read ``column-wise'' from the
rook-placement. This identification is helpful in the study of bi-de\-com\-po\-sability, because
bi-de\-com\-po\-sability is a symmetric property in $\pi$ and $\pi^{-1}$,
that is in rows and columns of the chessboard.

In this section, we define and study log-linear models for random permutations,
whose generators are partitions of the chessboard. More specifically,
we require that these partitions be the product of a row-partition and
a column-partition. A partition of the set $\uos{1}{n}$
into $s$ disjoint subsets (also called atoms) is given by
\begin{equation}\label{particio}
\mathcal{Z} = ( Z_1, \ldots, Z_s ): \ \cup_{i=1}^s Z_i = \uos{1}{n},\, Z_i \cap Z_j = \emptyset \; \forall i\neq j.
\end{equation}
If none of the sets $Z_i$ is empty, we call $s$ the size of the
partition. If of two such partitions, one partitions the set of rows,
the other the set of columns of the $n\times n$ chessboard, then the
result is a product partition of the board.

\begin{defi} A partition $\mathcal{B}$
 of the $n\times n$ chessboard is a \emph{product partition},
if there exist a partition
 $\mathcal{R}$ of size $r$ (called row-partition) and a partition
 $\mathcal{C}$ of size $c$ (called
 column-partition) of the set $\uos{1}{n}$ such
 that
\begin{equation}\label{racsos} \mathcal{B} = (B_{ij}): \ B_{ij} =
  R_i \times C_j =\{
 (x,y): x\in R_i, y\in C_j \}, \quad 1\leq i\leq r, \; 1\leq j\leq c. \end{equation}
We denote this by $\mathcal{B} = \mathcal{R} \times \mathcal{C}$.
\end{defi}

For any product partition $\mathcal{B}$ of the
$n\times n$ chessboard, we define the matrix-valued \emph{$\mathcal{B}$-marginal} function
$\pi \mapsto |\pi_{\mathcal{B}}|$ on $S_n$. For a permutation $\pi$, this statistic
gives the number of rooks falling into each $B_{ij}$ in
the rook-placement corresponding to $\pi$:
\begin{equation} \label{arnyek}
|\pi_{\mathcal{B}}|= (t_{ij}),\quad t_{ij} = |\{
1\leq s \leq n: (s,\pi(s))\in B_{ij} \}|. \end{equation}
In other words, if $\pi$ is a pairing between two labelled sets $A$
and $B$, then $|\pi_{\mathcal{B}}|$ is the $r\times c$ matrix
whose  $ij$th entry is the number of elements of
$A$ belonging to
$R_i$, which are paired with an element
of $B$ belonging to $C_j$.
The partition $\mathcal{B}$ of the chessboard gives rise to a partition
of $S_n$ via the fuction $\pi\mapsto|\pi_{\mathcal{B}}|$ : the permutations $\pi$ and $\sigma$ belong to the same atom
of this partition if and only if $|\pi_{\mathcal{B}}|=|\sigma_{\mathcal{B}}|.$
The subspace of $\mathbb{R}^{n!}$ spanned by the indicator vectors of these atoms will be
denoted by $U^{\mathcal{B}}.$ Equivalently
\begin{equation} \label{logalter}
U^{\mathcal{B}} = \{ v\in \mathbb{R}^{n!}: |\pi_{\mathcal{B}}|=|\sigma_{\mathcal{B}}|
\Rightarrow v(\pi)=v(\sigma) \}. \end{equation}
The vectors $v\in U^{\mathcal{B}}$ are just the functions $\pi\mapsto
v(\pi)$ on $S_n$, which are measurable with respect to the (atomic)
$\sigma$-algebra with atoms $\{ \pi: |\pi_{\mathcal{B}}| = (t_{ij})
\}$, where $(t_{ij})$ takes all possible values. We denote this
$\sigma$-algebra by $\sigma(\mathcal{B})$.

We will define a log-linear model by a set of product partitions,
called the generators of the model.
Of course, a similar definition is possible also with generator
partitions which are not of product form.
For the spanned subspace, we use the
notation $\Span(\cdot)$.

\begin{defi} \label{loglinmod} Let $\mathcal{B}_1, \ldots,
\mathcal{B}_s$ be product partitions of the chessboard, and use the
 simplifying notation $U^{\mathcal{B}_i} = U^i$. We say that $p$ belongs to the \emph{log-linear model generated
 by these partitions} if
\begin{equation} \log p(\pi) = \sum_{ i=1}^s \theta^i( |\pi_{\mathcal{B_i}}| ) \quad
\pi\in S_n, \end{equation}
where the $\theta^i$ functions are arbitrary parameters. Equivalently, we require
that
\begin{equation} \log p \in \Span(U^1, \ldots, U^s).
\end{equation}
We will use the notation $\mathcal{L}(\mathcal{B}_1, \ldots,
\mathcal{B}_s)$ for this model. 
\end{defi}

In the rest of this section, we give a sufficient condition, when the
intersection of two log-linear models is itself a log-linear model,
with directly identifyable generators. The
proofs can be found in Section \ref{appendix}. The first lemma
describes the relationship between conditional independence and
orthogonal intersection.

\begin{lem} \label{felt fugg} Let $(\Omega,\mathcal{A},P)$ be a probability space, and denote
  by $L_2(\mathcal{A})$ the Hilbert space of square-integrable random
  variables on it. For a $\sigma$-algebra $\mathcal{D}\subset \mathcal{A}$,
  denote by $L_2(\mathcal{D})$ the closed linear subspace of $L_2(\mathcal{
  A})$ consisting of all $\mathcal{D}$-measurable random variables.
  Let $\mathcal{D}_1$, $\mathcal{D}_2\subset \mathcal{A}$. Then $L_2(\mathcal{D}_1)\mme L_2(\mathcal{D}_2)$  
  if and only if $\mathcal{D}_1$ and $\mathcal{D}_2$ are conditionally
  independent, given $\mathcal{D}_1\cap \mathcal{D}_2$.
\end{lem}
There is a partial ordering on the set of partitions. Partition
$\mathcal{Z}=(Z_1,\ldots, Z_s)$ is finer that $\mathcal{W}=(W_1,
\ldots, W_t)$ (or
$\mathcal{W}$ is coarser than $\mathcal{Z}$) if for every $i$ there
exists a $j$ such that $Z_i \subset W_j$. Denote this by
$\mathcal{Z}\succ  \mathcal{W}$. Clearly, this implies
$U^{\mathcal{Z}}\supset  U^{\mathcal{W}}$.
By the application of Lemma \ref{felt fugg}, we get
\begin{lem} \label{fo1} Let
$\mathcal{R'}\succ \mathcal{R}$ and $\mathcal{C'}\succ
\mathcal{C}$ be partitions of $\uos{1}{n}$. Then we have
\begin{equation} \label{lemma3}
U^{\mathcal{R}\times \mathcal{C'}} \mme U^{\mathcal{R'}\times
 \mathcal{C}} \mbox{ and \ }
U^{\mathcal{R}\times \mathcal{C'}} \cap U^{\mathcal{R'}\times
 \mathcal{C}} = U^{\mathcal{R}\times \mathcal{C}}.
\end{equation}
\end{lem}

The next two lemmas formulate simple facts from linear algebra,
which will be needed in the sequel. We write $U = U_1
\oplus U_2$ for orthogonal decomposition, that is when $U =
\Span(U_1,U_2)$ and $U_1$ and $U_2$ are orthogonal. 

\begin{lem} \label{alg1}
Suppose that $U = \Span( U_i: i\in I)$, $V = \Span(V_j: j\in J)$ are two subspaces, and $U_i \mme V_j$ for every pair
$i,j$. Then $U\mme V$, and  $U\cap V = \Span( U_i \cap V_j: i\in I,
 j\in J)$. \end{lem}

\begin{lem} \label{alg2} Let $U=U_1 \oplus U_2$ and $V=V_1 \oplus V_2$ be
 two subspaces with orthogonal decompositions. If $U \mme V$, $U_1
 \mme V_1$, $U \mme V_1$, and $U_1 \mme V$ hold,
 then $U_2 \mme V_2$ is also true, and
\[ U\cap V = (U_1 \cap V_1) \oplus (U_1 \cap V_2) \oplus (U_2 \cap
 V_1) \oplus (U_2 \cap V_2). \]
\end{lem}

As a direct corollary of Lemma \ref{fo1} and Lemma \ref{alg1}, we obtain
\begin{corol} \label{loglin metszet}
Let $\mathcal{L}(\mathcal{R}_i \times \mathcal{C}:\,
  i=1,\ldots, s)$ and $\mathcal{L}(\mathcal{R}
\times \mathcal{C}_j:\, j=1,\ldots, t)$ be two log-linear models, and suppose that $\mathcal{R}\succ
\mathcal{R}_i$ and $\mathcal{C}\succ \mathcal{C}_j$ for all $1\leq
i\leq s$, $1\leq j\leq t$. Then the intersection of the two models is
the log-linear model $\mathcal{L}(\mathcal{R}_i \times \mathcal{C}_j:\,
  i=1,\ldots, s, \, j=1,\ldots, t)$.
\end{corol}

\subsection{De\-com\-po\-sability as a log-linear model}
In this section, we prove Theorem \ref{metszet}.
Recall from \eqref{sop} the definition of a consecutive partition of
$\uos{1}{n}$. A consecutive partition, which contains only two
neighboring sections is called a \emph{bold section}. That is, the
$k$th bold section, containing the sections $k-1$ and $k$, is given by
\begin{equation}\label{bold} \Phi_k = ( \uos{1}{k-1}, \{k \},
  \uos{k+1}{n} ), \quad 2\leq k \leq n-1.
\end{equation}
We will extend the notation $\Phi_k$ to $k=1$ and $k=n$ for the sake of
convenience.
The (consecutive) partition which partitions $\uos{1}{n}$ into $n$
sets is called the \emph{full partition}:
\begin{equation} \label{full}
\Psi = ( \{1\}, \{2\}, \ldots, \{n\}).\end{equation}

From the multiplicative form \eqref{s1.0}, it is straightforward that
the $L$-de\-com\-po\-sable exponential family $\mathcal{P}_L^+$ is the
log-linear model 
\begin{equation} \label{bkdef} \mathcal{P}_L^+ = 
\mathcal{L}( \Phi_k \times \Psi: 1\leq k\leq n). \end{equation}
That is, the generators are the products of bold sections with the
full partition. The reason is that the $0-1$ matrix $|\pi_{\Phi_k
  \times \Psi}|$ is equivalent to the vector of unordered marginals
$\{ \pi_{\Phi_k} \}$ defined in \eqref{unordered m}.
 
 A submodel which we will use in the sequel has as generators
 coarser partitions. A consecutive partition, which contains only one
 section is called a \emph{thin section}. Denote the $k$th thin
 section by
\begin{equation}\label{thin} \widetilde{\Phi}_k = ( \uos{1}{k}, 
  \uos{k+1}{n} ), \quad 1\leq k \leq n-1.
\end{equation}
Notice that $\Phi_1 = \widetilde{\Phi}_1$ and $\Phi_{n} =
\widetilde{\Phi}_{n-1}$. We extend the notation $\widetilde{\Phi}_k$
to $k=n$ for the sake of convenience. The log-linear model
\begin{equation} \label{rhullam} 
\mathcal{P}_{L_S}^+ =\mathcal{L}(\widetilde{\Phi}_k\times \Psi: 1\leq k\leq n-1)\end{equation}
is a submodel of the $L$-de\-com\-po\-sable family, which consists of positive distributions $p$ for which
the conditional probability $P(\Pi(|C|+1)= x \mid \Pi\uos{1}{|C|} = C)$
 depends on the pair $(x,C)$ only through their union, $C\cup \{x\}$,
where, as before, $\Pi$ is a random permutation with distribution $p$.
We will call these distributions \emph{$L_S$-de\-com\-po\-sable}, where $S$
 stands for ``set'', indicating that the choice of the $k$th element
 of the random permutation depends only on the set to be formed by the first $k$
 elements. We define \emph{$L'_S$-de\-com\-po\-sable} distributions similarly.

Recall from \eqref{logalter} the definition of the subspace
corresponding to a product partition, and for the sake of brevity introduce the notations
\begin{equation}
U^{\Phi_k \times \Psi} = U^k,\;  U^{\widetilde{\Phi}_k\times \Psi}= \widetilde{U}^k.
\end{equation}
From the partial ordering of partitions, we get that $\widetilde{U}^k \subset
 U^k$, denote the orthogonal complement of $\widetilde{U}^k$ in
 $U^k$ by $F^k$.
By the same argument, $\widetilde{U}^k \subset U^{k+1}$ also holds. This yields that
\[ \Span( U^k: 1\leq k \leq n)  = \Span( F^k: 1\leq k \leq
n, \widetilde{U}^n). \]
Since $\Phi_1 = \widetilde{\Phi}_1$, we get $F^1 = \{ \nulla
\}$. In addition, as $\widetilde{\Phi}_n$ is the trivial partition,
$\widetilde{U}^n = \Span( \egy)$. Thus the subspace belonging to the
$L$-de\-com\-po\-sable log-linear model is
\begin{equation} \label{fspan} F =  \Span( U^k: 1\leq k \leq n) = \Span( F^k: 2\leq k <
n, \egy ). \end{equation}
In the next lemma, we show that the subspaces on the righthandside of \eqref{fspan}
not only span $F$, but give an orthogonal decomposition. The proof is
found in Section \ref{appendix}.

\begin{lem}\label{felbontas} The subspaces $F^k$  $(2\leq k \leq n)$ are orthogonal
 to each other and to the vector $\egy$.
\end{lem}

The number of free parameters in the $L$-de\-com\-po\-sable exponential
family, which we denote by $b_n$, is the dimension of $F$ minus
one. From the orthogonal decomposition \eqref{fspan} of $F$ it is
immediate that
\begin{equation} b_n = \dimz (F)-1 = \sum_{k=2}^n \dimz (F^k) = \sum_{k=2}^n
 \binom{n}{k}(k-1) = 2^n(n/2 -1)+1,\end{equation}
where the dimension of $F^k$ is easy to calculate.

The decomposition \eqref{fspan} simplifies the calculations regarding
the dimension of the bi-de\-com\-po\-sable model as well. Interchanging the
role of rows and columns, we see that the $L'$-de\-com\-po\-sable loglinear
family is
\begin{equation} 
\mathcal{P}_{L'}^+ = \mathcal{L}( \Psi \times\Phi_k: 1\leq k\leq n). \end{equation}
Define the \emph{$(k,\ell)$th bold cross-section} as
\begin{equation} \label{vasker}
\mathcal{H}_{k\ell}= \Phi_k \times \Phi_{\ell},\end{equation}
where the component (row and column) partitions are bold sections
defined in \eqref{bold}. By Corollary \ref{loglin metszet}, the
bi-de\-com\-po\-sable log-linear model is given by
\begin{equation} \label{pbplusz}
\mathcal{P}_b^+ = \mathcal{L}( \mathcal{H}_{k\ell}: 1\leq k,\ell\leq
n). \end{equation}

In order to find the dimension of $H$, define the subspaces $V^{\ell}, \widetilde{V}^{\ell},
G^{\ell}$ in the $L'$-de\-com\-po\-sable model just as we defined $U^k, \widetilde{U}^k, F^k$
in the $L$-de\-com\-po\-sable model. Using the notation introduced in
\eqref{logalter}, 
\begin{equation} 
U^{\Psi \times \Phi_{\ell}} = V^{\ell},\;  U^{\Psi \times \widetilde{\Phi}_{\ell}}= \widetilde{V}^{\ell}.
\end{equation}

Applying Lemma \ref{felbontas}, the subspace corresponding
to the $L'$-de\-com\-po\-sable model can be written as
\begin{equation}\label{gspan} G = \oplus_{\ell =2}^n  G^{\ell} \oplus \Span(\egy ). \end{equation}

By Lemma \ref{fo1}, for any pair
\[ U \in\{U^k, \widetilde{U}^k:\ 2\leq k \leq n\}, \ V \in\{V^{\ell},\widetilde{V}^{\ell}: \
2\leq \ell \leq n\},\]
$U \mme V$, since these subspaces correspond to product
partitions, where in the $U$-partitions, the column-partition is as
fine as possible, and in the $V$-partitions, the row-partition is as
fine as possible. By Lemma \ref{alg2}, we get
$F^k \mme G^{\ell}$ for all $k,\ell$. By Lemma \ref{alg1},
the space $H=F\cap G$ corresponding to bi-de\-com\-po\-sable distributions has the
orthogonal decomposition
\begin{equation}\label{hspan} H =  \oplus_{2\leq k,\ell\leq n} (F^k
 \cap G^{\ell}) \oplus \egy.\end{equation}

It remains to find the dimension and a basis of $F^k \cap
G^{\ell}$. For the time being, fix $k$ and $\ell$.
By Lemma \ref{fo1}, the subspace corresponding to the $(k,\ell)$th bold cross-section $\mathcal{H}_{k\ell}$ is just $U^k
\cap V^{\ell}$.
Observe that $F^k \cap G^{\ell}$ consists of exactly those vectors of
the space $U^k \cap V^{\ell}$, which are orthogonal to both $\widetilde{U}^k$ and $\widetilde{V}^{\ell}$.
Recall that the $\pi$th coordinate of a vector in $U^k \cap V^{\ell}$ depends only on
its marginal $|\pi_{\mathcal{H}_{k\ell}}| = (t_{ij})_{1\leq i,j\leq 3} $ defined by \eqref{arnyek}, that is the number of rooks
$\pi$ places in the nine parts into which the $(k,\ell)$th bold cross-section
divides the chessboard.

As the nine elements of
the matrix $|\pi_{\mathcal{H}_{k\ell}}|$ must satisfy row-sum and column-sum
constraints, the vector is determined by its coordinates $t_{ij}$ for
$i,j=1,2$. Furthermore, all of $t_{12},t_{21}, t_{22}$ can be
either zero or one. We will specify the marginal $|\pi_{\mathcal{H}_{k\ell}}|$ by
two coordinates: $a= t_{11}+t_{12}+t_{21}+t_{22}$ and $q$, where $q$ codes the placement
of the rooks in the middle row and column of the $3\times 3$
partition. Our coding is as follows. In the $k$th row and $\ell$th
column, there is either one rook in the intersection of the row and
the column (code $5$), or there are two rooks, one on a horizontal, and one on
a vertical arm of the cross. In this latter case, the two occupied
arms point towards a plane-quarter, and we use the usual numbering
of the plane-quarters as coding (one or two
arms of the cross may be missing, but this does not cause any
problems). That is, for fixed $k,\ell$,
\begin{equation}  \label{TGbasis1}
a^{k\ell}(\pi) = \mid \{ i: \ 1\leq i \leq k,\; 1\leq \pi(i) \leq \ell
\} \mid,
\end{equation}
and \begin{equation} \label{TGbasis2}
q^{k\ell}(\pi) = \left\{ \begin{array}{llll}
1 & \mbox{if} & \pi(k)> \ell, &  \pi^{-1}(\ell)< k \\
2 & \mbox{if} & \pi(k)< \ell, &  \pi^{-1}(\ell)< k \\
3 & \mbox{if} & \pi(k)< \ell, &  \pi^{-1}(\ell)> k \\
4 & \mbox{if} & \pi(k)> \ell, &  \pi^{-1}(\ell)> k \\
5 & \mbox{if} & \pi(k)= \ell. & \\
\end{array} \right.
\end{equation}
Since the $\pi$th coordinate of a vector in $U^k \cap V^{\ell}$ depends only on
its marginal $|\pi_{\mathcal{H}_{k\ell}}|$, a basis of $U^k \cap
V^{\ell}$ is given by the indicator vectors of all the possible values
of this marginal. Therefore, for each $a,q$, we define this indicator vector
\begin{equation} \label{rhodef} \rho^{k\ell}_{aq} (\pi) = \chi \{ a^{k\ell}(\pi) = a,\; q^{k\ell}(\pi) = q \}.\end{equation}
Of course, for many pairs $a,q$, these are zero vectors. We now
determine all cases when $\rho^{k\ell}_{aq}$ is not identically
zero. First, we need
\begin{equation}\label{alegal} \max(0, k+\ell-n) \leq a \leq \min(k,\ell), \end{equation}
as there must be a non-negative number of rooks in each rectangle of
the board. $q$ can usually be anything from $1$ to $5$, except
\begin{equation} \label{tlegal}
\begin{array}{lcl}
a=0& \Rightarrow &q=4,\\
a=1& \Rightarrow &q\in \{1,3,4,5\},\\
a=k<j& \Rightarrow &q\in \{2,3,5\},\\
a=j<k& \Rightarrow &q\in \{1,2,5\},\\
a=j=k& \Rightarrow &q\in \{2,5\}.
\end{array}
\end{equation}
We call the pairs $a,q$ satisfying \eqref{alegal} and \eqref{tlegal}
non-trivial pairs. After all this preparation, we are ready for the
proof of Theorem \ref{metszet}.

\begin{proof}[Proof of Theorem \ref{metszet}]
We have to show \eqref{metszetdim}. The number of free parameters of
the bi-de\-com\-po\-sable log-linear model is $\dimz (H)-1$, since only those
vectors $v$ in $H$ are allowed for which $p=e^v$ is a probability
distribution. By \eqref{hspan}, we only need to determine the
dimension of each subspace $F^k\cap G^{\ell}$, which consists of the vectors of
$U^k \cap V^{\ell}$, which are orthogonal to both
$\widetilde{U}^k$ and $\widetilde{V}^{\ell}$.

Let $u = \sum_{a,q} c_{aq} \rho^{k\ell}_{aq}$ be an arbitrary vector in
$U^k \cap V^{\ell}$, we find when it is orthogonal to $\widetilde{U}^k$ and
$\widetilde{V}^{\ell}$. First take a vector
$v(\pi) = \chi ( \pi\uos{1}{k} = C )$ in the basis of $\widetilde{U}^k$,
and introduce the notation $\mid C\cap \uos{1}{ \ell} \mid =
a$.
With $h=(k-1)!(n-k)!$, the scalar product is calculated as
\[ (u,v) = \left\{ \begin{array}{lll}
c_{a1} (k-a)h + c_{a2} (a-1)h + c_{a5}h & \mbox{if} & \ell\in C \\
c_{a3} ah + c_{a4}(k-a)h & \mbox{if} & \ell\not\in C \\
\end{array} \right. \]

Similarly, if $v(\pi) = \chi( \pi^{-1}\uos{1}{ \ell} = D )$ is a
basis vector of $\widetilde{V}^{\ell}$, $\mid D\cap \uos{1}{ k } \mid =
a$, and $g=(\ell-1)!(n-\ell)!$, then
\[ (u,v) = \left\{ \begin{array}{lll}
c_{a3} (\ell-a)g + c_{a2} (a-1)g + c_{a5}g & \mbox{if} & k\in D \\
c_{a1} ag + c_{a4}(\ell-a)g & \mbox{if} & k\not\in D \\
\end{array} \right. \]

Thus $F^k \cap G^{\ell}$ consists of the linear combinations of those vectors
$\sum_{q=1}^5 c_{aq} \rho^{k\ell}_{aq}$ for which the
above four linear combinations of the coefficients $c_{aq}$ are
zero. Of the four constraints
on the coefficients, only three are linearly independent, so in most
cases there are two linearly independent solutions for the five
coefficients. The cases $a=0,1,\min(k,\ell)$ must be treated separately,
it is readily seen that in the case $a=0$, the only solution is
zero,  while in the cases $a=1, \min(k,\ell)$ there is one
non-zero solution. Let $\Delta^{k\ell}_a$ denote the number of linearly independent
solutions, that is $\Delta^{k\ell}_a$ is either zero, one or two.
The following vectors form an orthogonal basis of $F^k \cap G^{\ell}$
(with the exception that some vectors may be $\nulla$):
\begin{equation} \label{mubazis}
\begin{array}{ll}
\mu^{k\ell}_{a1} =& -\rho^{k\ell}_{a2} +
(a-1)\rho^{k\ell}_{a5}  \\
\mu^{k\ell}_{a2} =& -(\ell-a)a\rho^{k\ell}_{a1} + (k-a)(\ell-a)
\rho^{k\ell}_{a2} - (k-a)a\rho^{k\ell}_{a3} + \\
&+ a^2\rho^{k\ell}_{a4} +
(k-a)(\ell-a)\rho^{k\ell}_{a5} \\
\end{array}
\end{equation}

Finally, since
\[ \sum_{k,\ell} \dimz(F^k \cap G^{\ell}) = \sum_{i\geq 1} |\{ (k,\ell) : \dimz(F^k \cap G^{\ell}) \geq i \} |,\]
to finish the proof of the theorem, it suffices to show that for
$1\leq i \leq n$
\[ |\{ (k,\ell) : \dimz(F^k \cap G^{\ell}) \geq i \} | = (n-i)^2. \]

To this end, let us find those $k,\ell$, for which $\dimz(F^k \cap G^{\ell})\geq 2j+2$. This happens if
among the quantities $\Delta^{k\ell}_a$ there are either two $1$'s
and at least $j$ $2$'s, or one $1$ and at least $(j+1)$ $2$'s.

The first of these cases occurs when $\ell+k \leq n+1$ and $\min\{ k,\ell\}
\geq j+2$, while the second case occurs when $\ell+k \geq n+2$ and $\max\{ k,\ell\}
\leq n-j-1$.
But if $\ell+k \leq n+1$ and $k,\ell\geq j+2$, then $k,\ell
\leq n-j-1$ also holds. Similarly, if $\ell+k \geq n+2$ and $k,\ell
\leq n-j-1$, then at the same time $k,\ell\geq j+3>j+2$.
Therefore, $\dimz(F^k \cap G^{\ell})\geq 2j+2$ holds if and only if $j+2 \leq k,\ell \leq n-j-1$,
and there are $[n-(2j+2)]^2$ such pairs.

Let us find those $k,\ell$, for which $\dimz(F^k \cap G^{\ell})\geq 2j+1$. This happens if
among the quantities $\Delta^{k\ell}_a$ there are either two $1$'s
and at least $j$ $2$'s, or one $1$ and at least $j$ $2$'s.

The first of these cases occurs when $\ell+k \leq n+1$ and $\min\{ k,\ell\}
\geq j+2$, while the second case occurs when $\ell+k \geq n+2$ and $\max\{ k,\ell\}
\leq n-j$.
But if $\ell+k \leq n+1$ and $k,\ell\geq j+2$, then $k,\ell
\leq n-j-1<n-j$ also holds. Similarly, if $\ell+k \geq n+2$ and $k,\ell
\leq n-j$, then at the same time $k,\ell\geq j+2$.
Therefore, $\dimz(F^k \cap G^{\ell})\geq 2j+1$ holds if and only if $j+2 \leq k,\ell \leq n-j$,
and there are $[n-(2j+1)]^2$ such pairs.
\end{proof}

\begin{remark}
In \eqref{mubazis}, we found an orthogonal basis $\{ \egy,
\mu^{k\ell}_{ai} \}$ of the space $H$. This orthogonality is
convenient for finding the parameters corresponding to a bi-de\-com\-po\-sable distribution.
There exists a basis consisting of indicator vectors as well, as follows.
Denote for any $k,\ell,a$
$\nu^{k\ell}_a = \sum_{q=1}^5 \rho^{k\ell}_{aq}$,
where $\rho^{k\ell}_{aq}$ was defined in \eqref{rhodef}. That is,
$\nu^{k\ell}_a$ is the indicator vector of the event that there are exactly
$a$ rooks in the upper left $k\times \ell$ rectangle of the chessboard.
The following vectors, together with $\egy$, form a
 basis of $H$:
\begin{equation} \label{nubazis} \begin{array}{ll}
\nu^{k\ell}_a:  & 1\leq k,\ell \leq n-1,
\max(0, k+\ell-n) < a \leq \min(k,\ell), \\
\rho^{k\ell}_{a5}:  & 1\leq k,\ell \leq n-1,
\max(1, k+\ell-n) < a \leq \min(k,\ell). \end{array}\end{equation}
This statement can be proved by induction, we omit the somewhat
lengthy calculations.
\end{remark}

\begin{remark}\label{subset_r} Notice that the vectors $\nu^{k\ell}_a$, together with $\egy$ form the
basis of the subspace associated with those positive distributions,
called \emph{bi$_S$-de\-com\-po\-sable} distributions,
which belong to the intersection of the $L_S$- and $L'_S$-de\-com\-po\-sable
models. By Corollary \ref{loglin metszet}, this is again a log-linear model, with generating product partitions
\begin{equation}\label{vekker} \widetilde{\mathcal{H}}_{k\ell} = \widetilde{\Phi}_k \times \widetilde{\Phi}_{\ell},\end{equation}
which we call the \emph{thin $(k,\ell)$th cross-sections}, dividing the chessboard
into four rectangles. The component partitions were defined in
\eqref{thin}. Notice that in this case, with $k,\ell$ fixed, $\nu^{k\ell}_a$
($a$ takes on all its possible values) is an orthogonal basis of the
subspace corresponding to $\widetilde{\mathcal{H}}_{k\ell}$. These subspaces are
``almost linearly independent'' in the sence that the only linear
dependence is that they all contain the vector $\egy$.
From this it follows that the number of parameters in this model is
\begin{equation} e_n =\sum_{j=0}^{\lfloor (n-1)/2\rfloor }
  (n-2j-1)^2. \end{equation}
The vectors $\rho^{k\ell}_{a5}$ represent the ``difference'' between
the bi-de\-com\-po\-sable and the bi$_S$-de\-com\-po\-sable distributions.
\end{remark}

\begin{remark} We have calculated the number of free parameters in the
$L$-decompo\-sab\-le, bi-de\-com\-po\-sable, bi$_S$-de\-com\-po\-sable models. For the
 sake of completeness we mention that the
 number of free parameters in the remaining $L_S$-de\-com\-po\-sable
 model is given by
\begin{equation}
c_n = \sum_{i=1}^n \left[ \binom{n}{i} -1 \right] = 2^n - n -1.
\end{equation}
\end{remark}

\subsection{Maximum likelihood estimation}

As we have seen in the previous sections, the positive bi-de\-com\-po\-sable
distributions $\mathcal{P}_b^+$ on $S_n$ form an exponential family with $d_n$
parameters. 

Denote by $\Pi_1, \ldots, \Pi_m$ a sample taken from a positive
bi-de\-com\-po\-sable distribution, and let $r(\pi)$ stand for the relative frequency of the permutation $\pi$
in the sample. The maximum likelihood estimate of the
true distribution, or equivalently, of its parameters, does not appear
to have
an explicit form in general, the likelihood function has to be maximized by
numerical methods. The iterative proportional fitting procedure (IPFP), used
in the theory of log-linear models, is one option. This algorithm converges to the maximum
likelihood estimate, if it exists. We describe briefly the
implementation of this algorithm in our setting.

The generators of the bi-de\-com\-po\-sable log-linear model are the bold
 cross-sections
 $\mathcal{H}_{k\ell}$ in \eqref{vasker}, which define the
marginal functions $|\pi_{\mathcal{H}_{k\ell}}|$ given by \eqref{arnyek}. The maximum
likelihood estimate is a distribution $p^*\in \mathcal{P}_b^+$ such that the distributions
of the marginals $|\Pi_{\mathcal{H}_{k\ell}}|$  under $p^*$ are the same as
under the empirical distribution $r$. There is at most one such $p^*$
in $\mathcal{P}_b^+$. In some cases, the maximum likelihood estimate
does not exist, because no distribution in $\mathcal{P}_b^+$ gives the
same distribution of the marginals as the empirical
distribution (we say that the sample contains structural zeros). In
these cases, a suitable $p^*$ can only be found in the closure
$cl(\mathcal{P}_b^+)$.
Numerical studies indicate that $cl(\mathcal P^+_b) = \mathcal P_b$.

The IPFP algorithm proceeds by cyclically
fitting the distributions of the individual marginals $|\Pi_{\mathcal{H}_{k\ell}}|$
to that observed in the sample. It converges to the unique element in
$cl(\mathcal{P}_b^+)$ which agrees with the empirical distribution in
all marginals. Starting from an
arbitrary $p^1 \in \mathcal{P}_b^+$ (say the uniform distribution),
the $n$th iteration step calculates
\begin{equation}
\label{iter}
p^{n+1}(\pi) =
\frac{\sum_{\sigma: |\sigma_{\mathcal{H}_{k\ell}}| =|\pi_{\mathcal{H}_{k\ell}}|}
 r(\sigma)}{\sum_{\sigma: |\sigma_{\mathcal{H}_{k\ell}}| =|\pi_{\mathcal{H}_{k\ell}}|}
 p^{n}(\sigma)} p^n(\pi),\end{equation}
where the pair $(k,\ell)$ runs cyclically over all possible values.

\begin{remark} By Remark \ref{subset_r}, maximum likelihood estimation in the
 bi$_S$-de\-com\-po\-sable model proceeds in an analogous way, namely  by running the IPFP algorithm
 with $|\pi_{\widetilde{\mathcal{H}}_{k\ell}}|$ of
 \eqref{vekker} instead of $|\pi_{\mathcal{H}_{k\ell}}|$, i.e. we use
 the thin cross-sections instead of the bold ones.
\end{remark}

\begin{remark}
In $\mathcal{P}_L$ and in $\mathcal{P}_{L'}$,
the maximum likelihood estimate can be given explicitly. For example,
the $L$-de\-com\-po\-sable model is parametrized by the conditional
probabilities \eqref{s3}. The maximum likelihood estimate of these
conditional probabilities is given by the corresponding conditional
probabilities under the empirical distribution.

Numerical studies indicate that the maximum likelihood estimate in the
family $cl(\mathcal{P}_b^+)$ can also be obtained by iteratively
calculating the maximum likelihood
projections on the component spaces $\mathcal{P}_L$ and $\mathcal{P}_{L'}$. 
\end{remark}

\section{Examples}
In this section, we collect some models from the literature, which are
de\-com\-po\-sable in at least one way. These models are submodels of the
``free'' de\-com\-po\-sable models, since they place specific constraints on
the parameters.

\begin{pelda}[Order statistics models] Consider an experiment in
which people are asked to rank sounds according to their
loudness, say in increasing order. One might suppose that the actual perception of each
stimulus is a random variable, whose relative ordering determines the
person's ordering of the sounds. This example was studied by
Thurstone \cite{Thurstone27}, and later by Daniels
\cite{Daniels50}. If the random variable associated to the $i$th sound
is $X_i$ with continuous distribution $F_i$, then the resulting
distribution on the orderings is
\[ p(\pi) = P(X_{\pi(1)} < \cdots < X_{\pi(n)}).\]
The assumption that $X_i$ are independent leads to the so called \emph{order statistics models}. If the
distributions form a location family $F_i(x) = F(x-\mu_i)$, the model
is called \emph{Thurstone model}. A well-studied case is when $F(x) =
1 - \exp(-\exp x)$ is the Gumbel distribution. This is called
\emph{Luce model}, and is equivalent to the model derived by Luce on
the basis on his ranking postulate and Choice Axiom in
\cite{Luce59}. This model is $L$-de\-com\-po\-sable, with
canonical decomposition
\[  \Lambda(x,C) = \frac{\theta_x}{\sum_{y\not\in C} \theta_y},\]
where $\theta_y$ are arbitrary positive parameters (with sum equal to $1$) associated with the
objects. However, the model is not $L'$-de\-com\-po\-sable
for arbitrary $\theta.$
\end{pelda}

\begin{pelda}[Paired comparisons models] A model suggested by
Babington Smith \cite{BS50} creates the ordering of the $n$ objects by
making every possible paired comparison independently of each other. The result of such a
tournament can be represented by a directed graph: if the graph
contains no directed circle, then it corresponds to a unique ordering
of the objects. Conditioning on the event that the graph is
circle-free, we get
\[ p(\pi) = c(\theta) \prod_{i<j} \theta_{\pi(i)\pi(j)},\]
where $\theta_{xy}$ is the probability that object $x$ is preferred to
object $y$ in a paired comparison. This model is $L$-de\-com\-po\-sable, with
\[\Lambda(x,C) = \prod_{y\in C} \theta_{yx}.\]
However, the model is not $L'$-de\-com\-po\-sable
for arbitrary parameters.
\end{pelda}

\begin{pelda}[Mallows-Bradley-Terry model] A special case of the
paired comparison model is given by $\theta_{xy}=\frac{\alpha_x}{\alpha_x+\alpha_y}$. This
form of the paired comparison probabilities was suggested by Bradley
and Terry \cite{BT52}, and Mallows \cite{Mallows57} suggested using
these probabilities in the paired comparison model. The resulting
distribution on the orderings is given by
\[ p(\pi) = c(\alpha)\prod_{i=1}^n \alpha_{\pi(i)}^{n-i} = c(\alpha)\prod_{j=1}^n
\alpha_j^{n-\pi^{-1}(j)},\]
which is bi-de\-com\-po\-sable.
\end{pelda}

\begin{pelda}[Multistage ranking model] This model, investigated by
Fligner and Verducci \cite{FV88} supposes the candidates are numbered from $1$ to
$n$. The ranking takes place stepwise. In the $k$th step, the best $k-1$ ranks are already
given out. The $k$th best candidate is then chosen from the
remaining ones, but only the relative order of the remaining
candidates is taken into account. In particular,
if the remaining candidates are $j_1 < \cdots < j_{n-k+1}$, then
choose $j_i$ with probability $\theta(i,k)$, where $\theta(i,k)$ are parameters
satisfying $\sum_{i=1}^{n-k+1} \theta(i,k)=1$. It is easily seen that this
model is $L$-de\-com\-po\-sable with
\[\Lambda(x,C) =
\theta(|\overline{C}\cap\uos{1}{x}|,|C|+1).\]
The model is also $L'$-de\-com\-po\-sable.
\end{pelda}

\begin{pelda}[Repeated insertion model] This model, studied by
Doignon, Peke\v{c} and Regenwetter \cite{DPR04} assumes that the
 ordering is created by considering the candidates one after the other
(according to their fixed numbering), and inserting the current candidate into
the order already formed by the previous ones. More specifically, for
each $k$, we have insertion probabilities $\theta(i,k)$, $i=1,\ldots, k$
with sum $1$. For the $k$th candidate, there are $k$ possible places
where he or she can be inserted into the order of the first $k-1$ candidates: insert
him or her between the $(i-1)$st and $i$th with probability $\theta(i,k)$.
This model is a ``dual'' of the
multistage ranking model, in the sense that it can be described
similarly to it, by interchanging ``ranks'' and ``candidates'' (but
not in the sense that the resulting permutations are each other's
inverse). It is $L$-de\-com\-po\-sable with
\[\Lambda(x,C) = \theta(|C\cap \uos{1}{x}|+1,x),\]
and it is also $L'$-de\-com\-po\-sable.
\end{pelda}

\begin{pelda}[Quasi-independence log-linear model]
Let $\theta_{ij}, \ 1\leq i,j\leq n$, be the elements of an
arbitrary doubly stochastic matrix, and
\[ p(\pi) = c(\theta)\prod_{i=1}^n \theta_{i\pi(i)} = c(\theta)\prod_{j=1}^n
\theta_{\pi^{-1}(j)j}.\]
This distribution is by the above equation bi-de\-com\-po\-sable. Writing it
in the log-linear form
\begin{equation}\label{m1} \log p(\pi) = \alpha^{(0)} + \alpha^{(1)}_{\pi(1)} +
\alpha^{(2)}_{\pi(2)}+\cdots + \alpha^{(n)}_{\pi(n)}, \end{equation}
it states the quasi-independence of the variables $\pi(i)$, $1\leq
i\leq n$.

It is easy tho see that
for $n>3$ the random permutations belonging to the quasi-independence model
have the following property. For any partition $\mathcal{Z}=(Z_1,Z_2)$
of size $2$, as in \eqref{particio}, with $|Z_1|=2$, the ordered
marginals $\Pi_{Z_1}$ and $\Pi_{Z_2}$ are conditionally independent,
given the unordered marginals $\{ \Pi_{\mathcal{Z}} \}$.
This property is a generalization of $L$-de\-com\-po\-sability,
and we can prove that only the quasi-independent distributions 
possess it. Moreover, for these distributions, a similar property is
satisfied with arbitrary partitions $\mathcal{Z}$.
\end{pelda}

\section{Invariance under relabellings}
As we have noted already, de\-com\-po\-sability of a random pairing
between sets $A$ and $B$ depends on how we label the elements of the
two sets. Suppose that a labelling on both sets is fixed, and the
random pairing function $\Pi: A\to B$ with these labellings becomes
$\Pi_{orig}: \uos{1}{n} \to \uos{1}{n}$.
Suppose that we relabel the set $A$ according to the permutation
$\sigma\in S_n$, that is object with original label $i$ receives the new
label $\sigma(i)$. Similarly, relabel the set $B$ according to
$\rho\in S_n$. Denote the random pairing function $\Pi: A\to B$ with
these new labellings $\Pi_{new}: \uos{1}{n} \to \uos{1}{n}$. If with the
original labelling, the pair of $i\in A$ is $\pi(i)\in B$, then with
the new labelling, the pair of $\sigma(i)\in A$ is $\rho\pi(i)\in
B$. Therefore, for the
distributions $p_{orig}$ and $p_{new}$,,
\begin{equation}\label{inv1} p_{orig}(\pi)= P(\Pi_{orig}= \pi) =
P(\Pi_{new} = \rho\pi\sigma^{-1} )  = p_{new}(\rho\pi\sigma^{-1}).
\end{equation}
In this section we investigate whether $L$-de\-com\-po\-sability is
preserved after such relabellings or not.

\begin{defi} Let  $\phi: S_n \to S_n$ be a one to one mapping. Then
 for any distribution $p$, define $p_{\phi}(\pi) = p(\phi(\pi))$. We say
 that the family $\mathcal{P}$ of distributions is
 \emph{invariant} under $\phi$, if
\[\mathcal{P}_{\phi} = \{ p_{\phi}: p\in\mathcal{P} \} \subset \mathcal{P}.\]
\end{defi}

Let us introduce some notation. For any $\sigma\in S_n$, let
$\phi_{\circ\sigma}: \pi \mapsto \pi\sigma$,
$\phi_{\sigma\circ}: \pi \mapsto \sigma\pi$
be the right and left multiplications by $\sigma$. Denote by $\sigma_{(12)}$ the permutation
which exchanges $1$ and $2$ only, and by $\sigma_r$ the reversing
permutation which maps $k$ to $n+1-k$.

In the ranking situation $(ii)$ described in the Introduction, 
a model for a random ranking is called \emph{label-invariant}, if it is
invariant under relabellings of the objects. It is called
\emph{reversible}, if it is invariant under reversing of the
ranks. The concepts of label-invariance and reversibility were studied for some wide classes
of ranking models in \cite{CFV}.

\begin{tet} \label{atszamoz} $\mathcal{P}_L$ is invariant under left multiplications, and under the group of right
 multiplications generated by $\sigma_r$ and $\sigma_{(12)}$ (for
 $n\geq 4$, this group contains eight right
 multiplications, including the identity). The family is not
 invariant under any other right multiplication.
\end{tet}
\begin{proof}
Invariance under left-multiplications follows e.g. from Property 3 in
Definition \ref{s1}, as well as invariance under right multiplication
by $\sigma_r$, since  the Markov property is reversible. Invariance
under right multiplication by $\sigma_{(12)}$ can be checked directly
using Property 1 in Definition 1.

To show that the family is not invariant under other right
multiplications, we prove that for all
other permutations $\sigma$ there exists a positive bi-de\-com\-po\-sable distribution $p$, such that
$p_{\circ\sigma}$ is not $L$-de\-com\-po\-sable. The group of
right multiplications generated by $\phi_{\circ\sigma_r},
\phi_{\circ\sigma_{(12)}}$ are the multiplications by the following
permutations:
\begin{equation}\label{nyolc} id,\,\, \sigma_r,\,\, \sigma_{(12)},\,\, \sigma_r   \sigma_{(12)},\,\, \sigma_r
 \sigma_{(12)} \sigma_r,\,\, \sigma_{(12)} \sigma_r,\,\,
\sigma_r   \sigma_{(12)}  \sigma_r   \sigma_{(12)},\,\,
\sigma_{(12)}  \sigma_r   \sigma_{(12)}
\end{equation}
We will use the following property: let
$p$ be $L$-de\-com\-po\-sable. Suppose that the probability of $\pi_{11}$ and
$\pi_{22}$ is positive, and
$a$ is such that $\pi_{11}\uos{1}{a} =
\pi_{22}\uos{1}{a}$. Define the ``crossover'' permutations:
\[ \begin{array}{rclrcl}
\pi_{12}(k) &=& \left\{ \begin{array}{ccc}
\pi_{11}(k) & \mbox{if} & k\leq a \\
\pi_{22}(k) & \mbox{if} & k> a\\ \end{array} \right., &
\pi_{21}(k) &=& \left\{ \begin{array}{ccc}
\pi_{22}(k) & \mbox{if} & k\leq a \\
\pi_{11}(k) & \mbox{if} & k> a\\ \end{array} \right.. \\
\end{array} \]
Then
\begin{equation} \label{inv2} \frac{p(\pi_{11})}{p(\pi_{12})} =
\frac{p(\pi_{21})}{p(\pi_{22})}. \end{equation}

If $\sigma$ not a member of the permutations in \eqref{nyolc}, then neither
is its inverse, and there exists an $2\leq a \leq n-2$, such that
\[  \sigma^{-1}\uos{1}{a} \neq \uos{1}{a},  \uos{n-a+1}{n}.
\]
Let $a$ be such a number. Therefore there exist $c,e
\in \uos{1}{a}$ and $d,f \not\in \uos{1}{a}$, for which
\[ c^* = \sigma^{-1}(c) > \sigma^{-1}(d) = d^*, \quad
e^* = \sigma^{-1}(e) < \sigma^{-1}(f) = f^*. \]
For the numbers $\alpha, \beta, \gamma$ we say that $\alpha$ separates
$\beta$ and $\gamma$, if
$\beta<\alpha<\gamma$ or $\beta>\alpha>\gamma$. Now, if $d^*\geq f^*$,
then $d^*$ (and $f^*$ as well) separates $c^*$ and $e^*$. If
$d^*<f^*$, then either one of them separates $c^*$ and
$e^*$, or $c^*$ (and $e^*$ as well) separates $d^*$ and
$f^*$. Therefore, one of the following two cases holds:
\[ \begin{array}{ccrrc}
1. & \exists &c,e \in \uos{1}{a}, &d \not\in \uos{1}{a}:&
  d^*\mbox{ separates } c^*,e^*\\
2. & \exists &c \in \uos{1}{a}, &d,f \not\in \uos{1}{a}:&
  c^*\mbox{ separates } d^*,f^* \\
\end{array} \]

The two cases can be treated in the same way. Let us deal with the
first one! Let $f \not\in \uos{1}{a}, f\neq d$ be arbitrary,
with $f^* = \sigma^{-1}(f)$. Recall \eqref{rhodef}, and let $p = c(d^*) \exp\{ \rho^{d^*d^*}_{d^*5} \}$,
this is a positive bi-de\-com\-po\-sable distribution.
Let $\pi_{11} = \sigma^{-1}$, from which we obtain $\pi_{22}$ by
exchanging two pairs:
\[ \pi_{22}(c)=e^*, \pi_{22}(e)=c^*,
\pi_{22}(d)=f^*, \pi_{22}(f)=d^*. \]
Denote by $\pi_{12}$ and $\pi_{21}$ the crossover permutations.
For these four permutations,
$p_{\circ\sigma}$ does not satisfy \eqref{inv2}.
On the one hand, multiplying $\pi_{11}$ by $\sigma$ from the right, we
get the identity permutation, for which $\rho^{d^*d^*}_{d^*5}=1$. On the
other hand, for both $\pi_{12}\sigma$ and
$\pi_{21}\sigma$, $\rho^{d^*d^*}_{d^*5}=0$ , since for the first,
$d^*$ is not a fixed point, and for the second, there is an element
greater than $d^*$ among the first $d^*$ elements. This completes the
proof.
\end{proof}

\section{Discussion and application}
In this paper, we introduced log-linear models for random
permutations, whose generators are product partitions of the chessboard. Examples are the
$L$-de\-com\-po\-sable, $L_S$-de\-com\-po\-sable, bi-de\-com\-po\-sable, and
bi$_S$-de\-com\-po\-sable models.
In all of these cases, we determined the number of
parameters. We showed how to calculate
the  maximum likelihood estimate of the continuous parameters either
directly (for the $L$-de\-com\-po\-sable model) or by the iterative
proportional fitting algorithm (in the other models). The natural
order(s) implied by the models on the set(s) is either known,
or it also has to be estimated. We studied the extent to which this order can be determined in the
$L$-de\-com\-po\-sable and the bi-de\-com\-po\-sable models. There are many other statistical
questions of interest, which we did not address in this paper.
Another theoretically, and perhaps also practically important question
is the characterization of de\-com\-po\-sable distributions, if we do not
restrict ourselves to the strictly positive case.  

Finally, we fit our models to one of the most investigated ranking data in the literature, the
1980 election of the American Psychological Association (APA). This
organization elects a president each year by asking
its members to rank five candidates. In 1980,
5738 complete rankings were cast. APA chooses the winner by the
Hare system.
See Fishburn \cite{Fishburn} for a review of the advantages
and disadvantages of this system.

Analyses of
these data can be found, among others, in Chung and Marden
\cite{ChMar93}, Diaconis \cite{Diaconis89}, McCullagh \cite{McC93},
and Stern \cite{Stern93}. One characteristic feature of the data is
that the members of the association can be divided into three distinct groups:
the research psychologists (candidates $1$ and $3$ belong here),
clinical psychologists (their candidates are $4$ and $5$), and the
community psychologists, to whom $2$ belongs. The first two groups
represent the majority of the members. Not surprisingly, analysis
shows that each group tends to prefer its own candidates.

Chung and Marden \cite{ChMar93} fit orthogonal contrast models to the data.
Diaconis \cite{Diaconis89} uses this dataset to illustrate the method
of spectral analysis of ranked data, with many pointers to
literature. McCullagh \cite{McC93} fits log-linear models based on
inversions to the data. We emphasize here that we do not attempt to
provide a thorough analysis of the APA data, our aim is merely to
illustrate the fit of our models on a real dataset. We used the
ordering data, Table \ref{tablazat1}
shows the maximum of the log-likelihood function ($L$) and the chi-square value of the
goodness of fit, with the degrees of freedom in parentheses for all
models. In the last column, we gave the standardized statistic $U =
(GOF - df)/\sqrt{df}$. In all cases, where applicable, the results
appearing in Table \ref{tablazat1} correspond to the best right/left
relabelling of the original data, which we found by
an exhaustive search. As expected, the order on the ranks indicated by
the models coincides with the natural order from best to worst. All
models agree  on the natural order of the candidates as well, $4$ is
in the middle, with $\{1,3\}$ and $\{2,5\}$ on its two sides (notice
that there are eight permutations fitting this pattern, as stated in
Theorem \ref{atszamoz}). The best fit is provided by the
$L$-de\-com\-po\-sable model. The results indicate that the rankings violate
de\-com\-po\-sability (i.e. conditional independence relations) more than
the orderings. There are at least two ways to find better fitting
models. Firstly, we could check which conditional independences do not
hold. Models which assume decomposability at only some row-sections
and some column-sections also fit into the log-linear setting
described in this paper. While a general expression for the number of
free parameters in these wider models is probably intractable, it can
be calculated numerically in any particular case. Secondly, one could
try to reduce the number of parameters by selecting a log-linear
model with fewer generating partitions.   

\begin{table}[htb] 
\centering
\caption{Fit of log-linear models to APA ordering data}\label{tablazat1}
\begin{tabular}{lccc}
\hline
Model &  $L$ & GOF $(df)$ & $U$\\
\hline
saturated &  $-26612$ &-\\
$L$-de\-com\-po\-sable & $-26661$ & $98.9\, (70)$& 3.45\\
$L'$-de\-com\-po\-sable & $-26674$ & $126.5\, (70)$& 6.75\\
$L_S$-de\-com\-po\-sable & $-26684$  & $144.8\, (93)$&5.37\\
$L'_S$-de\-com\-po\-sable  & $-26697$ & $171.7\, (93)$&8.16\\
bi-de\-com\-po\-sable&  $-26687$ & $151.8 \,(89)$&6.66\\
bi$_S$-de\-com\-po\-sable& $-26701$ & $180.1\, (99)$&8.15\\
uniform &  $-27470$ & $2183.0\, (119)$& \\
\hline
\end{tabular}
\end{table}

\section{Proofs} \label{appendix}
\begin{proof}[Proof of Proposition \ref{teglafugg}] If the stated
  conditional independences hold, the random permutation is clearly
  bi-de\-com\-po\-sable. For the other direction, suppose $\Pi$ is
  bi-de\-com\-po\-sable. By $L$-de\-com\-po\-sability, given $ \{\Pi_{\sopp} \}$,
$\Pi_{\sopp_i}$ are conditionally independent. Conditioning on $\{
\Pi^{-1}_{\underline{\lambda}} \}$ as well does not ruin this
  independence, since the additional condition restricts the values of
  the $\Pi_{\sopp_i}$ one by one. Thus we proved that $\Pi_{\sopp_i}$
  are conditionally independent, given $ \{\Pi_{\sopp} \}$ and $\{
\Pi^{-1}_{\underline{\lambda}} \}$, and using $L'$-de\-com\-po\-sability,
  the same is true for $\Pi^{-1}_{\underline{\lambda}_j}$. Denote a
  condition by $E=\{ \{\Pi_{\sopp} \} = u, \{
\Pi^{-1}_{\underline{\lambda}} \}=v\} $, then
\[ P(\Pi=\pi|E) = \prod_{i=1}^s P(\Pi_{\sopp_i}=\pi_{\sopp_i}|E),\]
and since
  \[ \Pi_{\sopp_i} = ( \Pi_{\sopp_i \times \underline{\lambda}_j}: \,
1\leq j\leq t),\]
where $\Pi_{\sopp_i \times \underline{\lambda}_j}$ is a function of
$\Pi^{-1}_{\underline{\lambda}_j}$, also 
\[ P(\Pi_{\sopp_i}=\pi_{\sopp_i}|E) = \prod_{j=1}^t P(\Pi_{\sopp_i
  \times \underline{\lambda}_j}
= \pi_{\sopp_i \times \underline{\lambda}_j}|E) ,\]
which proves the lemma.
\end{proof} 

\begin{proof}[Proof of Lemma \ref{felt fugg}]
Since $L_2({\mathcal{D}_1}\cap{\mathcal{D}_2})=L_2({\mathcal{D}_1})\cap
L_2({\mathcal{D}_2})$, the spaces $L_2({\mathcal{D}_1})$ 
\'es $L_2({\mathcal{D}_2})$ intersect orthogonally if and only if for
any $f\in L_2({\mathcal{D}_1})$, $g\in L_2({\mathcal{D}_2})$,
$$E\left(\bigl[f-E(f\mid{\mathcal{D}_1}\cap{\mathcal{D}_2})\bigr]
\bigl[g-E(g\mid{\mathcal{D}_1}\cap{\mathcal{D}_2})\bigr]\right)=0.$$ 

If the conditional independence relation holds, then the following stronger
equality holds:
$$E\left(\bigl[f-E(f\mid{\mathcal{D}_1}\cap{\mathcal{D}_2})\bigr] 
\bigl[g-E(g\mid{\mathcal{D}_1}\cap{\mathcal{D}_2})\bigr]
\mid{\mathcal{D}_1}\cap{\mathcal{D}_2}\right)=0.$$

In the other direction, if the spaces intersect orthogonally, then let
$E_1\in{\mathcal{D}_1}$, $E_2\in{\mathcal{D}_2}$, and denote by $C$ the
event that
$$P(E_1\cap
E_2\mid{\mathcal{D}_1}\cap{\mathcal{D}_2})-
P(E_1\mid{\mathcal{D}_1}\cap{\mathcal{D}_2})P(E_2\mid{\mathcal{D}_1}
\cap{\mathcal{D}_2})>0.$$
With $f=\chi(E_1)\chi(C)$ and $g=\chi(E_2)\chi(C)$, 
\begin{multline*}
E\left(\bigl[f-E(f\mid{\mathcal{D}_1}\cap{\mathcal{D}_2})\bigr]
\bigl[g-E(g\mid{\mathcal{D}_1}\cap{\mathcal{D}_2})\bigr]\right)=\\
E\left(E(fg\mid{\mathcal{D}_1}\cap{\mathcal{D}_2})-E(f\mid{\mathcal{D}_1}\cap{\mathcal{D}_2}) 
E(g\mid{\mathcal{D}_1}\cap{\mathcal{D}_2})\right)=\\
=E\left(\chi(C)\bigl[P(E_1\cap E_2\mid{\mathcal{D}_1}\cap{\mathcal{D}_2})-P(E_1\mid{\mathcal{D}_1}\cap{\mathcal{D}_2}) 
P(E_2\mid{\mathcal{D}_1}\cap{\mathcal{D}_2})\bigr]\right)=0.
\end{multline*}
This is possible only if $P(E_1\cap
E_2\mid{\mathcal{D}_1}\cap{\mathcal{D}_2})-
P(E_1\mid{\mathcal{D}_1}\cap{\mathcal{D}_2})P(E_2\mid{\mathcal{D}_1}
\cap{\mathcal{D}_2})\le 0$ with probability $1$. The reverse
inequality is obtained similarly, thus $E_1$ and $E_2$ are
conditionally independent, given ${\mathcal{D}_1}\cap{\mathcal{D}_2}$. 
\end{proof}

\begin{proof}[Proof of Lemma \ref{fo1}]
We apply Lemma \ref{felt fugg} to $S_n$ endowed
with the uniform distribution. In this case, orthogonal intersection
in the $L_2$-space is equivalent to orhtogonal intersection in $\mathbb{R}^{n!}$. 
The second statement in \eqref{lemma3} holds, because it is easy to
check that $\sigma(\mathcal R'\times \mathcal C)\cap \sigma(\mathcal
R\times \mathcal C') = \sigma(\mathcal R\times \mathcal C)$.
Concerning the first one, we have to prove that $|\Pi_{\mathcal
  R'\times \mathcal C}|$ and $|\Pi_{\mathcal R\times \mathcal C'}|$
are conditionally independent, given $|\Pi_{\mathcal R\times \mathcal
  C}|$, if $\Pi$ is a uniformly distributed random permutation, which
is again easy to check.
\end{proof}

\begin{proof}[Proof of Lemma \ref{alg1}]
By supposition, $Pr_{V_j}U_i \subset U_i$ for every $i,j$, therefore
$Pr_{V_j}U \subset U$ for every $j$, i.e. $U$ intersects each $V_j$
orthogonally. Consequently, $Pr_UV_j \subset V_j$ for every $j$,
which yields $Pr_UV \subset V$, which was to be proved. On the other
hand, let $W =\Span( U_i \cap V_j: i\in I, j\in J)$. Then
$Pr_{V_j}U_i\subset W$, furthermore $Pr_UV_j =
Pr_{V_j}U\subset W$, which leads to $Pr_UV \subset W$.
\end{proof}

\begin{proof}[Proof of Lemma \ref{alg2}]  We use that $U$ and $V$ intersect orthogonally if
and only if the projection operators onto them commute, that is $Pr_U
Pr_V = Pr_V Pr_U$. Now
\begin{multline*} Pr_{U_2} Pr_{V_2} = (Pr_{U} - Pr_{U_1})(Pr_{V} - Pr_{V_1}) = \\
Pr_{U}Pr_{V} - Pr_{U_1}Pr_{V} - Pr_{U}Pr_{V_1} + Pr_{U_1}Pr_{V_1}, \end{multline*}
and by supposition the operators in all four terms commute, yielding
$Pr_{V_2} Pr_{U_2}$. The second statement follows from Lemma
\ref{alg1}.
\end{proof}

\begin{proof}[Proof of Lemma \ref{felbontas}]
We use again that orthogonality in the $L_2$-space is equivalent to
orhtogonality in $\mathbb{R}^{n!}$. In this proof, we use the notation
\[ \sigma_k = \sigma( \Phi_k \times \Psi), \quad \widetilde{\sigma}_k
= \sigma( \widetilde{\Phi}_k \times \Psi).\]
An element of $F^k$ is a
difference $f_1 = f - E(f\mid \widetilde{\sigma}_k)$, where $f$ is
$\sigma_k$-measurable. Orthogonality to $\egy$ means that
$E(f_1)=0$. For the other statement, let $g_1$ be an element of $F^j$,
where $j>k$. It is easy to check that under the uniform distribution, $\sigma_k$ and $\sigma_j$ are
conditionally independent, given $\widetilde{\sigma}_k$. Therefore,
\[ E(f_1 g_1) = E[ E(f_1 g_1 \mid \widetilde{\sigma}_k)] = 
E[ E(f_1 \mid \widetilde{\sigma}_k)E(g_1 \mid \widetilde{\sigma}_k)] = 0,\]
since $E(f_1 \mid \widetilde{\sigma}_k)=0$ with probability $1$. 
\end{proof}

\vspace{\fill}

\end{document}